\documentclass[a4paper,12pt]{amsart}
\usepackage[T1]{fontenc}
\usepackage[utf8]{inputenc}
\usepackage{lmodern}
\usepackage{amsfonts}
\usepackage{amsmath}
\usepackage{amssymb}
\usepackage{mathrsfs}
\usepackage[colorlinks=true,linkcolor=blue,citecolor=blue,urlcolor=blue]{hyperref}
\usepackage{enumitem}
\numberwithin{equation}{section}
\theoremstyle{plain}
\newtheorem{thm}{Theorem}[section]
\newtheorem{prop}[thm]{Proposition}
\newtheorem{cor}[thm]{Corollary}
\newtheorem{lemma}[thm]{Lemma}

\newtheorem{property}[thm]{Property}
\theoremstyle{definition}
\newtheorem{defi}[thm]{Definition}
\newtheorem{rem}[thm]{Remark}
\newtheorem{ex}[thm]{Example}

\newcommand{\R}{\mathbb R}
\newcommand{\Rd}{\mathbb R^{d}}
\newcommand{\Rdo}{\mathbb R^{d}_{0}}
\newcommand{\Ro}{\mathbb R_{0}}
\newcommand{\cB}{\mathcal B}
\newcommand{\dd}{[d]}

\hypersetup{
  pdftitle={Product Pitt inequalities and signed Boas equivalences for general monotone functions},
  pdfauthor={Niyaz Tokmagambetov},
  pdfkeywords={Pitt inequality, Fourier transform, Hardy variation, general monotone function, product weight}
}

\begin{document}
\title[Product Pitt inequalities and signed Boas equivalences]{Product Pitt inequalities and signed Boas equivalences for general monotone functions}
\author[N. Tokmagambetov]{Niyaz Tokmagambetov}
\address{Niyaz Tokmagambetov:
  \endgraf
  Institute of Mathematics and Mathematical Modeling
  \endgraf
  28 Shevchenko str., 050010 Almaty, Kazakhstan
  \endgraf
  {\it E-mail address:} {\rm tokmagambetov@math.kz} }
\date{}
\keywords{Pitt's inequality; Fourier transform; Bounded variation in the sense of Hardy; General monotone functions; Product weights}
\subjclass[2020]{Primary 42B10; Secondary 26D10, 26B30, 42B35, 46E30}

\begin{abstract}
We prove product-weight Pitt inequalities for real or complex functions on $\mathbb R^d$
whose full mixed Vitali variation satisfies an exterior-tail general monotonicity condition.
A two-parameter condition gives spatial weights with different powers near the coordinate
hyperplanes and at infinity. For the classical parameters, a reverse estimate yields
a Boas norm equivalence without sign, parity or separability assumptions.
The equivalence holds for finite spatial norm and, for $L^1$ data, characterizes finiteness
of either norm; the unrestricted nonintegrable converse remains open.
We also prove an equivalence under primitive control, give explicit examples,
and establish sharpness of the Fourier-side exponent range.
\end{abstract}

\maketitle

\section{Introduction}\label{SEC:intro}
We study weighted Fourier inequalities for real or complex functions on $\Rd$, $d\ge1$,
whose full mixed variation satisfies an exterior-tail general monotonicity condition.
For $f\in L^1(\Rd)$ our normalization is
$\widehat f(\xi)=\int_{\Rd}f(x)e^{ix\cdot\xi}dx$; Section~\ref{SEC:FT}
defines the improper transform used for nonintegrable data.
Write $w_{\boldsymbol a}(x)=\prod_i|x_i|^{a_i}$.

Theorem~\ref{TH:main} proves, for $1<p\le q<\infty$ and
$-1/q<\alpha_i<1-1/q$,
\begin{equation}\label{EQ:main-intro}
 \|w_{\vec\alpha}\widehat f\|_q
 \lesssim\left\|\prod_i|x_i|^{\beta_i(x_i)}f\right\|_p,
 \qquad
 \beta_i(t)=2-\alpha_i-1/p-1/q-
 \begin{cases}\tau_i,&|t|<1,\\\kappa_i,&|t|\ge1,\end{cases}
\end{equation}
for $f\in GM(\Rdo;\vec\tau,\vec\kappa)$, $\tau_i\le\kappa_i$, with finite
right-hand side. Definition~\ref{DEFI:GMd} controls the full mixed Vitali variation
over exterior product regions by the weighted tail integral of $|f|$ with kernel
$\prod_i(|y_i|^{\tau_i}+|y_i|^{\kappa_i})^{-1}$.
The finite spatial norm ensures local integrability across the coordinate
hyperplanes and existence of the improper transform.

Our second main result is the signed Boas equivalence
\begin{equation}\label{EQ:Boas-intro}
 \|w_{\vec\alpha}\widehat f\|_p\asymp\|w_{\vec\beta}f\|_p,
 \qquad \beta_i=1-2/p-\alpha_i,
\end{equation}
for $1<p<\infty$, $-1/p<\alpha_i<1-1/p$, and
$f\in GM(\Rdo;\vec1,\vec1)$ with finite spatial norm
(Theorem~\ref{rev:main} and Corollary~\ref{rev:twosided}).
No sign, parity or separability assumption is imposed. For $L^1$ data,
finiteness of either norm characterizes finiteness of the other; the unrestricted
nonintegrable converse remains open (Remark~\ref{rev:comments}).
Theorem~\ref{rev:GMbar-thm} gives the same equivalence for a primitive-controlled class.

Weighted Fourier inequalities have a substantial history
\cite{Hei84,JS84,Muc83a,Muc83b,BH92,BH03,SW87}.
The classical Pitt inequality \cite{Pit37,Ste56,Bec95}, reviewed in
\cite[pp. 790--791]{DGT17}, yields the product-weight estimate for Schwartz functions
when
\begin{equation}\label{EQ:prod-range}
 -1/q<\alpha_i\le0,\qquad 0\le\beta_i<1/p',\qquad
 \alpha_i+\beta_i=1-1/p-1/q,
\end{equation}
where $p'=p/(p-1)$; tensorization uses Minkowski's inequality and $p\le q$.
For our classical parameters, \eqref{EQ:main-intro} retains this balance but removes
the sign restrictions on the weight exponents.

The general-monotone theory of Liflyand and Tikhonov
\cite{LT08,LT09,LT11} replaces monotonicity by variation bounds.
The usual local condition \cite[Equation (1.6)]{LT12} implies, by dyadic summation,
\begin{equation}\label{EQ:GM-classical}
 \int_x^\infty|df(y)|\lesssim\int_{x/m}^\infty\frac{|f(y)|}{y}\,dy.
\end{equation}
We use this exterior-tail condition without asserting equivalence to the local one.
Related one-dimensional and radial Pitt--Boas results appear in
\cite{LT08,GLT11,DGT13,Deb19} and \cite[Corollary 3.9]{LT12}.
The two-parameter one-dimensional class was introduced in
\cite[Definition 2.2]{Tok25}; our forward theorem extends
\cite[Theorem 3.5]{Tok25} to full mixed variation in arbitrary dimension.
Hardy variation originates in \cite{Har06,CA33}; its Fourier theory is developed in
\cite{Lif19}. Product Fourier averaging is studied in \cite{DNTW23}, and related
discrete GM results and counterexamples in \cite{Nur00,Tik07,DT09,Tik21,Oga23,DO23}.

Two-sided estimates for nonnegative GM functions with general weights are known
\cite[Theorem 3.1]{LT12}; signed one-dimensional Fourier results appear in
\cite[Corollary 1.3 and Section 5.3]{Deb19}, and anisotropic monotone results in
\cite{Muk17}. The contribution of \eqref{EQ:Boas-intro} is the specified full
mixed-variation tail class, including complex, nonseparable data.
The mixed-tail framework of \cite{Tok26} instead controls integrated tail variation
by interior rectangular integrals and gives mixed-norm estimates for even extensions.
Remark~\ref{REM:tailcomparison} separates our condition from that framework in one direction.
For $\tau_i\le1\le\kappa_i$, the refined spatial norm in \eqref{EQ:main-intro}
is equivalent to the classical one on the fixed GM class
(Proposition~\ref{PROP:normequiv}); it does not enlarge the finite-norm domain there.

The forward proof recovers partial variations from full variation, splits the transform
into $2^d$ interior/exterior pieces, and applies scalar Hardy inequalities.
The reverse proof pairs $\widehat f$ with rectangle indicators and absorbs mixed-variation
oscillation errors over a geometric grid. Truncation justifies the pairing for improper
transforms; a separate polynomial bound handles $L^1$ data before spatial finiteness is known.
Throughout, $A\lesssim B$ means $A\le CB$, with $C$ depending only on the stated
parameters and class constants, and $A\asymp B$ means both inequalities.

\section{Functions of bounded Hardy variation}\label{SEC:HBV}
\subsection{Notation}
We write $\dd:=\{1,\dots,d\}$, $\Ro:=\R\setminus\{0\}$ and $\Rdo:=\{x\in\Rd: x_{1}\cdots x_{d}\neq0\}$. For $t\in\R$ we put $\cB(t):=\{s\in\R:|s|<|t|\}$ and $\cB^{c}(t):=\{s\in\R:|s|\geq|t|\}$. For $S\subseteq\dd$ and $x\in\Rd$ we write $x_{S}:=(x_{i})_{i\in S}$, $S^{c}:=\dd\setminus S$, and we identify $x$ with $(x_{S},x_{S^{c}})$. For $x\in\Rdo$ and $S\subseteq\dd$ we define the sets
$$
Q_{S}(x):=\prod_{i=1}^{d}E_{i}, \qquad E_{i}:=\Biggl\{\begin{array}{cc} \cB^{c}(x_{i}), & i\in S;\\ \cB(x_{i}), & i\notin S,\end{array}
$$
and abbreviate $Q(x):=Q_{\emptyset}(x)$ and $Q^{c}(x):=Q_{\dd}(x)$. For $d>1$, $Q^c(x)$ denotes this exterior product set, not the set-theoretic complement of $Q(x)$. For $S\subseteq\dd$ we also write $Q(x_{S}):=\prod_{i\in S}\cB(x_{i})$ and $Q^{c}(x_{S}):=\prod_{i\in S}\cB^{c}(x_{i})$, regarded as subsets of $\R^{S}$. The set $Q^{c}(x)$ is the union of the $2^{d}$ closed \emph{corners}
$$
C_{\sigma}(x):=\prod_{i=1}^{d}\big\{y_{i}:\ \sigma_{i}y_{i}\geq|x_{i}|\big\}, \qquad \sigma\in\{-1,1\}^{d}.
$$
Finally, for $m>0$ and $x\in\Rd$ we write $mx:=(mx_{1},\dots,mx_{d})$, $x/m:=(x_{1}/m,\dots,x_{d}/m)$, and for $x\in\Rdo$, $x^{-1}:=(x_{1}^{-1},\dots,x_{d}^{-1})$.

\subsection{Mixed differences and Vitali variation}
Let $f:\Rdo\to\mathbb C$ be measurable. For $\emptyset\neq T\subseteq\dd$, a box $R_{T}=\prod_{i\in T}[a_{i},b_{i}]$ avoiding the coordinate hyperplanes and $z\in\R^{T^{c}}$ with nonzero coordinates, the mixed difference of $f$ over $R_{T}$ in the variables $T$, with the remaining variables frozen at $z$, is
$$
\Delta_{T}(f;R_{T};z):=\sum_{\epsilon\in\{0,1\}^{T}}(-1)^{|T|-|\epsilon|}\,f\big(c^{\epsilon},z\big), \qquad c^{\epsilon}_{i}:=\Biggl\{\begin{array}{cc} b_{i}, & \epsilon_{i}=1;\\ a_{i}, & \epsilon_{i}=0,\end{array}
$$
where $|\epsilon|=\sum_{i\in T}\epsilon_{i}$; for $T=\emptyset$ we set $\Delta_{\emptyset}(f;\cdot;z):=f(z)$. Two elementary properties of mixed differences will be used repeatedly:
\begin{enumerate}
\item[$(\Delta1)$] (additivity) if $R_{T}$ is divided by a hyperplane $\{y_{k}=c\}$, $k\in T$, into two boxes $R'_{T}$ and $R''_{T}$, then $\Delta_{T}(f;R_{T};z)=\Delta_{T}(f;R'_{T};z)+\Delta_{T}(f;R''_{T};z)$;
\item[$(\Delta2)$] (reduction) if $k\notin T$ and $z=(c,z')$ with $c$ the $k$-th coordinate, then
$$
\Delta_{T\cup\{k\}}\big(f;R_{T}\times[c,c'];z'\big)=\Delta_{T}\big(f;R_{T};(c',z')\big)-\Delta_{T}\big(f;R_{T};(c,z')\big).
$$
\end{enumerate}
The \emph{Vitali variation} of $f$ on $R_{T}$ in the variables $T$ (with $z$ frozen) is
$$
\int_{R_{T}}|d_{T}f(\cdot,z)|:=\sup\sum_{\mathbf j}\big|\Delta_{T}(f;R_{T,\mathbf j};z)\big|,
$$
the supremum being taken over all decompositions of $R_{T}$ into boxes $R_{T,\mathbf j}$ generated by partitions of the intervals $[a_{i},b_{i}]$, $i\in T$. By $(\Delta1)$ the sum does not decrease under refinement, and the Vitali variation is additive with respect to decompositions of $R_{T}$ into boxes and monotone with respect to inclusion. For an unbounded box $B_{T}=\prod_{i\in T}\{y_{i}:\sigma_{i}y_{i}\geq a_{i}\}$ we set $\int_{B_{T}}|d_{T}f(\cdot,z)|:=\sup\int_{R_{T}}|d_{T}f(\cdot,z)|$ over bounded boxes $R_{T}\subset B_{T}$, and for a union of finitely many such boxes with disjoint interiors we sum the variations. In particular, for $x\in\Rdo$ and $T\subseteq\dd$ we write
$$
V_{T}(f;x):=\int\limits_{Q^{c}(x_{T})}|d_{T}f(\cdot,x_{T^{c}})| \quad (T\neq\emptyset), \qquad V_{\emptyset}(f;x):=|f(x)|,
$$
and $V(f;x):=V_{\dd}(f;x)$.
Thus $V(f;x)$ is the sum of the Vitali variations of $f$ over the $2^{d}$ corners $C_{\sigma}(x)$, it depends only on $(|x_{1}|,\dots,|x_{d}|)$, and it is non-increasing in each $|x_{i}|$.
\begin{defi}\label{DEFI:HBV}
We say that a measurable $f:\Rdo\to\mathbb C$ belongs to $HBV_{0}(\Rdo)$ if
\begin{enumerate}
\item[$(a)$] $V(f;x)<\infty$ for every $x\in\Rdo$;
\item[$(b)$] for every $k\in\dd$ and every $z\in\R^{\dd\setminus\{k\}}$ with non-zero coordinates, $f(x)\to0$ as $|x_{k}|\to\infty$ with $x_{\dd\setminus\{k\}}=z$.
\end{enumerate}
\end{defi}
For $d=1$, $HBV_{0}(\Ro)$ is the class $BV_{0}(\Ro)$: functions of bounded variation on $\{|x|\geq a\}$ for every $a>0$ which vanish at infinity. The vanishing condition $(b)$ is essential: it allows one to recover $f$ and all its partial mixed variations from the full mixed variation.
\begin{property}\label{PROP:recover}
Let $f\in HBV_{0}(\Rdo)$, $x\in\Rdo$ and $T\subseteq\dd$. Then
$$
V_{T}(f;x)\leq V(f;x).
$$
In particular, $|f(x)|$ is bounded by the variation on the exterior corner containing $x$, and hence $|f(x)|\leq V(f;x)$ and, for every $S\subseteq\dd$ and every $z\in\R^{S^{c}}$ with non-zero coordinates, the function $x_{S}\mapsto f(x_{S},z)$ belongs to $HBV_{0}(\R^{S}_{0})$.
\end{property}
\begin{proof}
It suffices to prove the cornerwise estimate
\begin{equation}\label{EQ:recover-corner}
 \int_{C_T(x)}|d_Tf(\cdot,x_{T^c})|
 \le\int_{C(x)}|d_{\dd}f|,
 \quad C(x)=\prod_i[x_i,\infty),\quad C_T(x)=\prod_{i\in T}[x_i,\infty),
\end{equation}
for $x>0$; reflections and summation give the assertion.
Argue downward in $|T|$, the case $T=\dd$ being immediate.
Choose $k\notin T$ and put $T'=T\cup\{k\}$. For a grid of cells $R_{T,\mathbf j}$
in a bounded box of $C_T(x)$, write
$g_{\mathbf j}(t)=\Delta_T(f;R_{T,\mathbf j};(t,x_{T^c\setminus\{k\}}))$.
Reduction $(\Delta2)$, followed by the induction hypothesis, gives
\[
 \sum_{\mathbf j}|g_{\mathbf j}(x_k)|
 \le\sum_{\mathbf j}|g_{\mathbf j}(Y)|
 +\int_{C_{T'}(x)}|d_{T'}f(\cdot,x_{T'^c})|
 \le\sum_{\mathbf j}|g_{\mathbf j}(Y)|+\int_{C(x)}|d_{\dd}f|.
\]
Each term $g_{\mathbf j}(Y)$ tends to zero by the coordinate-tail condition.
Let $Y\to\infty$ and take the supremum over grids and bounded boxes.
The same argument includes $T=\emptyset$; the section assertion follows immediately.
\end{proof}
The class $HBV_{0}(\Rdo)$ is not contained in $L^{1}(\Rd)$, nor in $L^{1}_{loc}(\Rd)$: no integrability near the coordinate hyperplanes is imposed. Integrability will be a separate hypothesis in Section \ref{SEC:main}.

\begin{lemma}[Vanishing variation tails]\label{LM:Vtail}
For $f\in HBV_0(\Rdo)$, with all other coordinates fixed and nonzero,
$V(f;x)\to0$ as $|x_j|\to\infty$.
\end{lemma}
\begin{proof}
Fix a corner above a positive cutoff, and partition its $j$th ray into consecutive
bounded intervals tending to infinity. Variation is additive over any finite
collection of the resulting slabs, so their variations have finite sum, bounded
by the variation on the corner. Every bounded rectangle lying beyond the $n$th
cut can be subdivided into finitely many such slabs. Its variation is bounded by
the corresponding tail of that convergent sum. Taking the supremum over bounded
rectangles and then summing over the finitely many corners proves the assertion.
\end{proof}

\section{General monotone functions of several variables}\label{SEC:GM}
Let $\vec\tau=(\tau_{1},\dots,\tau_{d})$ and $\vec\kappa=(\kappa_{1},\dots,\kappa_{d})$ with $\tau_{i}\leq\kappa_{i}$ for all $i$, and put
$$
w_{i}(t):=|t|^{\tau_{i}}+|t|^{\kappa_{i}}, \qquad t\in\Ro, \qquad W(y):=\prod_{i=1}^{d}w_{i}(y_{i}), \qquad y\in\Rdo.
$$
Note that $w_{i}(t)\backsimeq|t|^{\tau_{i}}$ for $|t|<1$ and $w_{i}(t)\backsimeq|t|^{\kappa_{i}}$ for $|t|\geq1$.
\begin{defi}\label{DEFI:GMd}
We say that $f\in HBV_{0}(\Rdo)$ is \emph{general monotone with parameters $\vec\tau$ and $\vec\kappa$}, and write $f\in GM(\Rdo;\vec\tau,\vec\kappa)$, if there exist $C>0$ and $m\geq1$ such that
\begin{equation}\label{EQ:GMd}
V(f;x)=\int\limits_{Q^{c}(x)}|d_{\dd}f| \leq C\int\limits_{Q^{c}(x/m)}\frac{|f(y)|}{W(y)}\,dy =: C\, I(f;x/m) \qquad \hbox{for all } x\in\Rdo.
\end{equation}
\end{defi}
For $d=1$ we denote this class by $GM(\Ro;\tau,\kappa)$, and $GM(\Ro;1,1)$ is the two-sided version of the tail condition \eqref{EQ:GM-classical}, which follows from the usual local GM condition of \cite{LT08, LT11, LT12}. The right-hand side of \eqref{EQ:GMd} may be infinite, in which case the condition is void; under the hypotheses of Theorem \ref{TH:main} this does not happen (Remark \ref{REM:finite}).

\begin{ex}[Products]\label{EX:products}
If $g_i\in GM(\Ro;\tau_i,\kappa_i)$, then $f(x)=\prod_i g_i(x_i)$ belongs to
$GM(\Rdo;\vec\tau,\vec\kappa)$: mixed differences factor, so
$V(f;x)=\prod_iV(g_i;x_i)$ and $I(f;x)=\prod_iI(g_i;x_i)$.
Take $m=\max_i m_i$. The zero-product case is immediate; all integral
factorizations are understood in the extended nonnegative sense.
Section~\ref{SEC:1d} supplies explicit factors.
\end{ex}

\begin{ex}[Hardy-monotone functions]\label{EX:monotone}
Let $f\ge0$ satisfy the coordinate-tail condition in Definition~\ref{DEFI:HBV} and
\[
 (-1)^d\sigma_1\cdots\sigma_d\Delta_{\dd}(f;R)\ge0
\]
on every box in the orthant with signs $\sigma$.
Telescoping, with successive coordinate limits at infinity, gives variation
$f(\sigma|x|)$ on that corner and shows decrease in each absolute coordinate.
Hence $V(f;x)=\sum_\sigma f(\sigma|x|)$, and for $m>1$,
\[
 I(f;x/m)\ge\int_{\{|x_i|/m\le|y_i|\le|x_i|\}}
 \frac{f(y)}{2^d\prod_i|y_i|}\,dy
 \ge 2^{-d}(\log m)^dV(f;x).
\]
Thus $f\in GM(\Rdo;\vec1,\vec1)$.
The nonproduct example $(1+\sum_i|x_i|)^{-a}$, $a>0$, satisfies the sign
condition because its full mixed derivative on each orthant has that sign.
\end{ex}

\begin{ex}[Signed nonseparable data]\label{EX:signed}
In dimension two the function
\[
 f(x_1,x_2)=(x_1-x_2)\mathbf1_{[1,2]^2}(x_1,x_2)
\]
changes sign and is not a product of one-variable functions. It belongs to
every parameter class in Definition~\ref{DEFI:GMd}. Indeed its full variation
is bounded, as follows either by writing it as the difference of two products
of bounded-variation functions or by computing its mixed differences.
The variation vanishes if any cutoff exceeds $2$. For the remaining cutoffs,
$m=3$ ensures that the defining integral includes $[1,2]^2$ and is bounded
below by the positive integral of $|x_1-x_2|/W(x)$ on that box.
Its spatial weighted norms are finite for all finite exponents. Thus both
the forward theorem and the signed Boas equivalence apply directly to
nonseparable, sign-changing data. Multiplication by further box indicators
gives the same conclusion in dimensions $d>2$.
\end{ex}

\begin{property}\label{PROP:incl}
Let $\tau_{2,i}\leq\tau_{1,i}\leq\kappa_{1,i}\leq\kappa_{2,i}$ for $i=1,\dots,d$. Then
$$
GM(\Rdo;\vec\tau_{2},\vec\kappa_{2})\subseteq GM(\Rdo;\vec\tau_{1},\vec\kappa_{1}).
$$
\end{property}
\begin{proof}
For $|y_i|<1$, $|y_i|^{-\tau_{2,i}}\le|y_i|^{-\tau_{1,i}}$;
for $|y_i|\ge1$, $|y_i|^{-\kappa_{2,i}}\le|y_i|^{-\kappa_{1,i}}$.
Using the two power comparisons for $w_i$ gives $W_2^{-1}\lesssim W_1^{-1}$,
so $I_2\lesssim I_1$ and the GM bound transfers.
\end{proof}
\begin{prop}[Independent dilations]\label{REM:dilation}
Let $f$ satisfy \eqref{EQ:GMd} with $C,m$ and let $\lambda_i>0$. Then $f_{\vec\lambda}(x)=f(\lambda_1x_1,\ldots,\lambda_dx_d)$ satisfies the same condition with $m$ and
\[
 C_{\vec\lambda}=C\prod_{i=1}^d\max\{\lambda_i^{1-\tau_i},\lambda_i^{1-\kappa_i}\}.
\]
In particular membership is preserved under every fixed independent dilation.
\end{prop}
\begin{proof}
The mixed variation of $f_{\vec\lambda}$ above $x$ is that of $f$ above $\vec\lambda x$. Substitute $z_i=\lambda_i y_i$ in the latter GM bound and use
\[
 \frac{\lambda_i}{\lambda_i^{\tau_i}|y_i|^{\tau_i}+\lambda_i^{\kappa_i}|y_i|^{\kappa_i}}
 \leq\frac{\max\{\lambda_i^{1-\tau_i},\lambda_i^{1-\kappa_i}\}}{w_i(y_i)}.
\]
This proves the assertion. The displayed bound is uniform over all independent dilations when every $\tau_i=\kappa_i=1$. Outside this case it need not be uniform; no claim about the optimal constant for an individual function is made.
\end{proof}

\section{Pointwise estimates of the Fourier transform}\label{SEC:FT}
For $f\in HBV_{0}(\Rdo)$, $\xi\in\Rdo$ and $S\subseteq\dd$ we define
\begin{equation}\label{EQ:pieces}
\hat f_{S}(\xi):=\int\limits_{Q(\xi^{-1}_{S^{c}})}\Bigg(\ \int\limits_{Q^{c}(\xi^{-1}_{S})} f(x_{S},x_{S^{c}})\,e^{ix_{S}\cdot\xi_{S}}\,dx_{S}\Bigg)e^{ix_{S^{c}}\cdot\xi_{S^{c}}}\,dx_{S^{c}},
\end{equation}
where the inner integral is the sum over the $2^{|S|}$ corners of $Q^{c}(\xi^{-1}_{S})$ of iterated improper Riemann integrals (their existence is part of Lemma \ref{LM:ibp} below), and the outer integral is a Lebesgue integral over the bounded box $Q(\xi^{-1}_{S^{c}})$. For $S=\emptyset$ the inner integral is absent and $\hat f_{\emptyset}(\xi)=\int_{Q(\xi^{-1})}f(x)e^{ix\cdot\xi}dx$; for $S=\dd$ the outer integral is absent. Whenever the $\hat f_{S}$ are well defined we set
\begin{equation}\label{EQ:FT-def}
\hat f:=\sum_{S\subseteq\dd}\hat f_{S},
\end{equation}
which agrees with the Fourier transform of $f$ whenever $f\in L^{1}(\Rd)$.

We first prove the basic integration by parts estimate on a single corner; for $d=1$ this is the usual bounded-variation integration-by-parts argument; see \cite{LT08, LT09}.
\begin{lemma}\label{LM:ibp}
Let $f\in HBV_{0}(\Rdo)$, $a\in(0,\infty)^{d}$, $C(a)=\prod_{i}[a_{i},\infty)$ and $\xi\in\Rdo$. Then the iterated improper integral $\int_{C(a)}f(x)e^{ix\cdot\xi}dx$ (integrating successively in $x_{1},\dots,x_{d}$) exists and
\begin{equation}\label{EQ:ibp}
\Big|\int_{C(a)} f(x)\,e^{ix\cdot\xi}\,dx\Big|\leq\frac{2^{d}}{|\xi_{1}\cdots\xi_{d}|}\int_{C(a)}|d_{\dd}f|.
\end{equation}
The same holds for each of the other $2^{d}-1$ corners.
\end{lemma}
\begin{proof}
For $d=1$, integration by parts on $[a,A]$ gives
\[
 \int_a^Af(x)e^{ix\xi}dx
 =\left[\frac{f(x)e^{ix\xi}}{i\xi}\right]_a^A
 -\frac1{i\xi}\int_a^Ae^{ix\xi}\,df(x).
\]
The upper boundary vanishes and the Stieltjes integral converges absolutely.
Since $|f(a)|\le\int_a^\infty|df|$, the bound is $2|\xi|^{-1}\int_a^\infty|df|$.

Inductively, put $x=(x',t)$, $a=(a',a_d)$ and
$F(t)=\int_{C(a')}f(x',t)e^{ix'\cdot\xi'}dx'$.
Property~\ref{PROP:recover} makes each section admissible.
Apply the induction hypothesis to $f(\cdot,t')-f(\cdot,t)$ and use $(\Delta2)$:
\[
 |F(t')-F(t)|\le\frac{2^{d-1}}{\prod_{i<d}|\xi_i|}
 \int_{C(a')\times[t,t']}|d_{\dd}f|.
\]
Summing over partitions yields
\begin{equation}\label{EQ:varF}
 \int_{a_d}^\infty|dF|\le\frac{2^{d-1}}{\prod_{i<d}|\xi_i|}
 \int_{C(a)}|d_{\dd}f|,\qquad F(t)\to0.
\end{equation}
The limit follows from the section estimate, cornerwise recovery, and
Lemma~\ref{LM:Vtail}. Apply the one-dimensional bound to $F$ to obtain
\eqref{EQ:ibp}. Coordinate reflections handle the other corners.
\end{proof}
\begin{lemma}\label{LM:pointwise}
Let $f\in HBV_{0}(\Rdo)$, $\xi\in\Rdo$ and $S\subseteq\dd$. Then the inner integral in \eqref{EQ:pieces} exists for every $x_{S^{c}}\in Q(\xi^{-1}_{S^{c}})$ with non-zero coordinates, and
\begin{equation}\label{EQ:pointwise}
\begin{split}
|\hat f_{S}(\xi)|&\leq\frac{2^{|S|}}{\prod_{i\in S}|\xi_{i}|}\int\limits_{Q(\xi^{-1}_{S^{c}})}V_{S}\big(f;(\xi^{-1}_{S},x_{S^{c}})\big)\,dx_{S^{c}}\\
&\leq\frac{2^{|S|}}{\prod_{i\in S}|\xi_{i}|}\int\limits_{Q(\xi^{-1}_{S^{c}})}V\big(f;(\xi^{-1}_{S},x_{S^{c}})\big)\,dx_{S^{c}},
\end{split}
\end{equation}
provided the last integral is finite (for $S=\emptyset$ this reads $|\hat f_{\emptyset}(\xi)|\leq\int_{Q(\xi^{-1})}|f(x)|dx$, and for $S=\dd$ the outer integral is absent).
\end{lemma}
\begin{proof}
Apply Lemma~\ref{LM:ibp} to the section in the variables $S$ and sum over its
corners. Integrating the remaining variables gives the first inequality;
Property~\ref{PROP:recover} gives the second. For $S=\emptyset$ use $|f|\le V$ directly.
\end{proof}
The following elementary observation will be used to convert the values $V(f;\xi^{-1})$ appearing in \eqref{EQ:pointwise} into integrals to which the Hardy inequalities apply.
\begin{lemma}\label{LM:mono}
Let $\phi:\Ro\to[0,\infty]$ be even and non-increasing in $|t|$. Then, for every $a>0$,
\begin{equation}\label{EQ:mono}
\phi(2a)\leq\frac{1}{2\ln2}\int\limits_{\cB^{c}(a)}\frac{\phi(t)}{|t|}\,dt=\frac{1}{2\ln2}\int\limits_{\cB(1/(2a))}\frac{\phi\big((2s)^{-1}\big)}{|s|}\,ds.
\end{equation}
\end{lemma}
\begin{proof}
Since $\phi(t)\geq\phi(2a)$ for $a\leq|t|\leq2a$ and $\int_{a\leq|t|\leq2a}|t|^{-1}dt=2\ln2$, the inequality follows; the equality is the substitution $t=(2s)^{-1}$ on each half-line.
\end{proof}

\section{Iterated Hardy inequalities}\label{SEC:Hardy}
We recall the one-dimensional Hardy inequalities with power weights in the form we need; they follow from the weight characterizations in \cite{KP03}. For $a,b\in\R$, $1<p\leq q<\infty$ and $g\geq0$ on $\R$,
\begin{align}
\left(\int_{\R}|\xi|^{a}\Big(\int\limits_{\cB(\xi)}g(x)dx\Big)^{q}d\xi\right)^{1/q}&\lesssim\left(\int_{\R}|x|^{b}g(x)^{p}dx\right)^{1/p} \qquad \hbox{if } a+1<0, \label{EQ:H1}\\
\left(\int_{\R}|\xi|^{a}\Big(\int\limits_{\cB^{c}(\xi)}g(x)dx\Big)^{q}d\xi\right)^{1/q}&\lesssim\left(\int_{\R}|x|^{b}g(x)^{p}dx\right)^{1/p} \qquad \hbox{if } a+1>0, \label{EQ:H2}
\end{align}
where in both cases
\begin{equation}\label{EQ:b}
b=\frac{p}{q}(a+1)+p-1.
\end{equation}

Equivalently, logarithmic coordinates turn these scalar operators into convolution
with one-sided exponential kernels; Young's inequality yields the same bounds.

For $\vec\epsilon\in\{0,1\}^{d}$ and $g\geq0$ on $\Rd$ we define the operators
$$
(T_{\vec\epsilon}\,g)(\xi):=\int\limits_{Q_{S(\vec\epsilon)}(\xi)}g(x)\,dx, \qquad S(\vec\epsilon):=\{i\in\dd:\epsilon_{i}=1\}, \qquad \xi\in\Rdo,
$$
so that the $i$-th variable is integrated over $\cB(\xi_{i})$ if $\epsilon_{i}=0$ and over $\cB^{c}(\xi_{i})$ if $\epsilon_{i}=1$.
\begin{lemma}\label{LM:Hardy-d}
Let $1<p\leq q<\infty$, $\vec\epsilon\in\{0,1\}^{d}$ and $a_{1},\dots,a_{d}\in\R$ be such that
$$
a_{i}+1<0 \ \hbox{ if } \ \epsilon_{i}=0, \qquad a_{i}+1>0 \ \hbox{ if } \ \epsilon_{i}=1, \qquad i\in\dd,
$$
and put $b_{i}:=\frac{p}{q}(a_{i}+1)+p-1$. Then, for all $g\geq0$ on $\Rd$,
\begin{equation}\label{EQ:Hardy-d}
\left(\int_{\Rd}\prod_{i=1}^{d}|\xi_{i}|^{a_{i}}\,\big(T_{\vec\epsilon}\,g(\xi)\big)^{q}d\xi\right)^{1/q}\lesssim\left(\int_{\Rd}\prod_{i=1}^{d}|x_{i}|^{b_{i}}\,g(x)^{p}dx\right)^{1/p}.
\end{equation}
\end{lemma}
\begin{proof}
Induct on $d$, starting from \eqref{EQ:H1}--\eqref{EQ:H2}.
Let $G(x_1,\xi')$ be the $(d-1)$-variable operator applied to $g(x_1,\cdot)$.
The scalar estimate in the first coordinate and Minkowski's inequality,
with $q/p\ge1$, give
\[
 \|T_{\vec\epsilon}g\|_{L^q(\prod_i|\xi_i|^{a_i})}^{p}
 \lesssim\int_{\R}|x_1|^{b_1}
 \|G(x_1,\cdot)\|_{L^q(\prod_{i>1}|\xi_i|^{a_i})}^{p}dx_1.
\]
Apply the induction hypothesis to $G(x_1,\cdot)$ and integrate in $x_1$.
\end{proof}

\section{The forward Pitt inequality}\label{SEC:main}
Let $1<p\leq q<\infty$ and $\vec\alpha=(\alpha_{1},\dots,\alpha_{d})\in\Rd$. For $i\in\dd$ we define
\begin{equation}\label{EQ:beta}
\beta_{i}(t):=-\alpha_{i}-\frac1q-\frac1p+2-\Biggl\{\begin{array}{cc} \tau_{i}, & |t|<1;\\ \kappa_{i}, & |t|\geq1,\end{array} \qquad t\in\Ro,
\end{equation}
and we write
$$
\mathcal R(f):=\left(\int_{\Rd}\prod_{i=1}^{d}|x_{i}|^{\beta_{i}(x_{i})p}\,|f(x)|^{p}dx\right)^{1/p}.
$$
We first check that finiteness of $\mathcal R(f)$ makes everything in Section \ref{SEC:FT} well defined.
\begin{rem}\label{REM:finite}
Let $-1/q<\alpha_{i}<1-1/q$ for all $i$, let $f\in GM(\Rdo;\vec\tau,\vec\kappa)$ and $\mathcal R(f)<\infty$. By H\"older's inequality with the product weight,
\begin{equation}\label{EQ:Holder}
\begin{split}
I(f;x)&\leq\mathcal R(f)\prod_{i=1}^{d}\left(\int\limits_{\cB^{c}(x_{i})}\frac{dt}{\big(|t|^{\beta_{i}(t)}w_{i}(t)\big)^{p'}}\right)^{1/p'}\\
&\lesssim\mathcal R(f)\prod_{i=1}^{d}\big(1+|x_{i}|^{\alpha_{i}+\frac1q-1}\big), \qquad x\in\Rdo.
\end{split}
\end{equation}
Indeed, $\beta_i(t)+\tau_i=\gamma_i$ for $|t|<1$ and $\beta_i(t)+\kappa_i=\gamma_i$ for $|t|\geq1$, where $\gamma_i=2-\alpha_i-1/q-1/p$, and $\gamma_{i}p'>1$ is equivalent to $\alpha_{i}<1-1/q$; hence $\int_{|t|\geq1}|t|^{-\gamma_{i}p'}dt<\infty$ and $\int_{|x_{i}|\leq|t|<1}|t|^{-\gamma_{i}p'}dt\lesssim|x_{i}|^{1-\gamma_{i}p'}$, and $(1-\gamma_{i}p')/p'=\alpha_{i}+\frac1q-1$. Consequently the right-hand side of \eqref{EQ:GMd} is finite for every $x\in\Rdo$, so that
$$
V(f;x)\lesssim\mathcal R(f)\prod_{i=1}^{d}\big(1+|x_{i}|^{\alpha_{i}+\frac1q-1}\big).
$$
Since $\alpha_{i}+\frac1q-1>-1$, the right-hand side is integrable over any bounded box in any subset of the variables. In particular the last integral in \eqref{EQ:pointwise} is finite for every $S$ and every $\xi\in\Rdo$, so that all the pieces $\hat f_{S}$, and hence $\hat f$ in \eqref{EQ:FT-def}, are well defined. Moreover $|f|\leq V$ proves $f\in L^1_{loc}(\Rd)$ automatically, including across the coordinate hyperplanes (whose assigned values do not affect integrals).
\end{rem}
\begin{thm}\label{TH:main}
Let $d\geq1$, $1<p\leq q<\infty$, and $\tau_{i}\leq\kappa_{i}$ for $i\in\dd$. Let $f\in GM(\Rdo;\vec\tau,\vec\kappa)$ with $\mathcal R(f)<\infty$. If
$$
-\frac1q<\alpha_{i}<1-\frac1q, \qquad i=1,\dots,d,
$$
then
\begin{equation}\label{EQ:main}
\left(\int_{\Rd}\prod_{i=1}^{d}|\xi_{i}|^{\alpha_{i}q}\,|\hat f(\xi)|^{q}d\xi\right)^{1/q}\lesssim\left(\int_{\Rd}\prod_{i=1}^{d}|x_{i}|^{\beta_{i}(x_{i})p}\,|f(x)|^{p}dx\right)^{1/p},
\end{equation}
with $\beta_{i}$ given by \eqref{EQ:beta}. The implied constant depends on $d,p,q,\vec\alpha,\vec\tau,\vec\kappa$ and on the constants $C$, $m$ in \eqref{EQ:GMd}, but not on $f$.
\end{thm}
The weight on the right-hand side of \eqref{EQ:main} takes $2^{d}$ different forms on the regions $\prod_{i}\{|x_{i}|\lessgtr1\}$. For $\vec\tau=\vec\kappa=\vec1$ it reduces to the product weight $\prod_{i}|x_{i}|^{\beta_{i}p}$ with $\alpha_{i}+\beta_{i}=1-\frac1p-\frac1q$, see Corollary \ref{COR:classical}.
\begin{proof}
Write $V(u)=V(f;u)$ and $b_i=1-\alpha_i-1/p-1/q$.
\emph{Step 1: from $\widehat f$ to $V$.}
By Lemmas~\ref{LM:pointwise} and \ref{LM:mono},
\begin{align*}
 |\widehat f_S(\xi)|&\lesssim\prod_{i\in S}|\xi_i|^{-1}
 \int_{Q(\xi_S)\times Q(\xi_{S^c}^{-1})}g_S(s_S,x_{S^c})\,ds_Sdx_{S^c},\\
 g_S(s_S,x_{S^c})&=\prod_{i\in S}|s_i|^{-1}V((2s_S)^{-1},x_{S^c}).
\end{align*}
Invert the frequencies outside $S$. Lemma~\ref{LM:Hardy-d}, with head integration
in every variable, applies with
\[
 a_i=\begin{cases}(\alpha_i-1)q,&i\in S,\\-\alpha_iq-2,&i\notin S.\end{cases}
\]
Indeed $a_i+1<0$ is respectively $\alpha_i<1-1/q$ and $\alpha_i>-1/q$.
After absorbing the factors $|s_i|^{-p}$, the input exponent for $i\in S$ is
$(\alpha_i-1)p+p/q-1$. Substituting $s_i=(2u_i)^{-1}$ turns it into
$-\alpha_ip-p/q+p-1=pb_i$, also the exponent for $i\notin S$.
Thus for every $S$,
\begin{equation}\label{EQ:step1}
 \|w_{\vec\alpha}\widehat f_S\|_q
 \lesssim\left\|\prod_i|u_i|^{b_i}V(u)\right\|_p.
\end{equation}
\emph{Step 2: from $V$ to $f$.}
The GM bound, dilation $u=mv$, and Lemma~\ref{LM:Hardy-d} with tail integration
in every variable and equal exponents $p$ give
\[
 \left\|\prod_i|u_i|^{b_i}V(u)\right\|_p
 \lesssim\left\|\prod_i|v_i|^{b_i}I(f;v)\right\|_p
 \lesssim\left\|\prod_i\frac{|x_i|^{b_i+1}}{w_i(x_i)}f(x)\right\|_p.
\]
The Hardy condition $pb_i+1>0$ is $\alpha_i<1-1/q$.
Since $b_i+1-\tau_i$ and $b_i+1-\kappa_i$ are the two values of
$\beta_i$, the last norm is comparable to $\mathcal R(f)$.
Sum \eqref{EQ:step1} over $S$ to conclude.
\end{proof}
\begin{cor}\label{COR:classical}
Let $1<p\leq q<\infty$, $-1/q<\alpha_{i}<1-1/q$ and $\alpha_{i}+\beta_{i}=1-\frac1p-\frac1q$ for $i\in\dd$. Then, for every $f\in GM(\Rdo;\vec1,\vec1)$ with finite right-hand side,
$$
\left(\int_{\Rd}\prod_{i=1}^{d}|\xi_{i}|^{\alpha_{i}q}\,|\hat f(\xi)|^{q}d\xi\right)^{1/q}\lesssim\left(\int_{\Rd}\prod_{i=1}^{d}|x_{i}|^{\beta_{i}p}\,|f(x)|^{p}dx\right)^{1/p}.
$$
In particular, this holds for all Hardy-monotone functions of Example \ref{EX:monotone}.
\end{cor}

For $p=q$, this is a weighted Hardy--Littlewood estimate.

\begin{cor}\label{COR:support}
Let the hypotheses of Theorem \ref{TH:main} be satisfied.
\begin{enumerate}
\item[$(i)$] If $\operatorname{supp}f\subset Q(\vec1)=\{|x_{i}|<1,\ i\in\dd\}$, then \eqref{EQ:main} holds with $\beta_{i}=-\alpha_{i}-\frac1q-\frac1p+2-\tau_{i}$.
\item[$(ii)$] If $\operatorname{supp}f\subset Q^{c}(\vec1)=\{|x_{i}|\geq1,\ i\in\dd\}$, then \eqref{EQ:main} holds with $\beta_{i}=-\alpha_{i}-\frac1q-\frac1p+2-\kappa_{i}$.
\item[$(iii)$] If $\tau_{i}=\kappa_{i}$ for all $i$, then $\beta_{i}(t)\equiv-\alpha_{i}-\frac1q-\frac1p+2-\tau_{i}$ is constant.
\end{enumerate}
\end{cor}
\begin{prop}[Equivalence of spatial norms]\label{PROP:normequiv}
Let $1<p\leq q<\infty$, $-1/q<\alpha_i<1-1/q$, and put $b_i=1-1/p-1/q-\alpha_i$. For any $f\in GM(\Rdo;\vec\tau,\vec\kappa)$ with $\mathcal R(f)<\infty$,
\begin{equation}\label{EQ:spatial-control}
 \left\|\prod_i|x_i|^{b_i}f\right\|_p\lesssim\mathcal R(f).
\end{equation}
If $\tau_i\leq1\leq\kappa_i$ for every $i$, then
\[
 \mathcal R(f)\leq\left\|\prod_i|x_i|^{b_i}f\right\|_p\lesssim\mathcal R(f).
\]
The constants depend on the parameters and the GM constants.
\end{prop}
\begin{proof}
By Property \ref{PROP:recover}, $|f(x)|\leq V(f;x)$. Step 2 of the proof of Theorem \ref{TH:main} controls the weighted variation with exponent $pb_i$. That step therefore gives
\[
 \left\|\prod_i|x_i|^{b_i}f\right\|_p
 \leq\left\|\prod_i|x_i|^{b_i}V(f;x)\right\|_p
 \lesssim\mathcal R(f),
\]
proving \eqref{EQ:spatial-control}. If the parameters straddle one, then $\beta_i(t)-b_i=1-\tau_i\geq0$ for $|t|<1$ and $\beta_i(t)-b_i=1-\kappa_i\leq0$ otherwise. Thus $|t|^{\beta_i(t)}\leq|t|^{b_i}$, proving the reverse inequality. In particular, there is no function in this class with finite refined norm and infinite classical spatial norm.
\end{proof}

\section{Signed Boas equivalences}\label{SEC:rev}

Throughout $1<p<\infty$, $-1/p<\alpha_j<1-1/p$ and $\beta_j:=1-2/p-\alpha_j$; then also
$-1/p<\beta_j<1-1/p$, so that $|\xi_j|^{\alpha_jp}$ and $|x_j|^{\beta_jp}$ are Muckenhoupt $A_p$
weights on $\R$ --- the reason every Hardy step below falls in range. We write
$w_{\vec\alpha}(\xi):=\prod_j|\xi_j|^{\alpha_j}$, and similarly $w_{\vec\beta}$.
Set
\begin{equation}\label{rev:Phi}
 \Phi(x)=\int_{Q^c(x)}\frac{|f(y)|}{\prod_j|y_j|}\,dy.
\end{equation}
For classical parameters $I(f;x)=2^{-d}\Phi(x)$; absorbing $2^{-d}$ into the
constant, the GM hypothesis becomes
\begin{equation}\label{rev:GM}
 V(f;x)\le C\Phi(x/m).
\end{equation}
Enlarging $m$ if necessary, we take $m>1$. All constants below refer to fixed
$C,m$. For a general $HBV_0$ function we use the standing condition
\begin{equation}\label{rev:standing}
 \int_{Q(b_{S^c})}V(f;(a_S,x_{S^c}))\,dx_{S^c}<\infty
 \quad(\emptyset\ne S\subsetneq[d],\ a_S,b_{S^c}>0).
\end{equation}
Together with local integrability, it makes every piece in \eqref{EQ:pieces}
well defined. Conditions indexed by empty collections are omitted.

The following isolates what the proof of Theorem~\ref{TH:main} gives before general
monotonicity is used.

\begin{prop}\label{rev:forwardV}
Let $f\in HBV_0(\R^d_0)\cap L^1_{loc}(\R^d)$ satisfy \eqref{rev:standing} and let
$\hat f=\sum_S\hat f_S$. Then $\|w_{\vec\alpha}\hat f\|_p\le K_0\|w_{\vec\beta}V(f;\cdot)\|_p$,
with $K_0$ depending only on $d,p,\vec\alpha$.
\end{prop}
\begin{proof}
Step 1 of the proof of Theorem~\ref{TH:main} uses no GM hypothesis.
With $q=p$, its weight exponent satisfies $pb_j=-\alpha_jp+p-2=\beta_jp$.
Apply that estimate to each piece and sum over the $2^d$ subsets.
\end{proof}

\subsection{Signed averages and the Fourier transform}
\begin{lemma}\label{rev:parseval}
Let $f\in L^1(\R^d)$ with $\|w_{\vec\alpha}\hat f\|_p<\infty$. Then for every box
$R=\prod_j[a_j,b_j]$,
\begin{equation}\label{rev:Fbound}
\Big|\int_Rf\Big|\ \le\ \frac1{(2\pi)^d}\int_{\R^d}|\hat f(\eta)|\prod_{j=1}^d\min\Big(b_j-a_j,\frac2{|\eta_j|}\Big)d\eta,
\end{equation}
and the right-hand side is finite. In particular, with
\begin{equation}\label{rev:calG}
\mathcal G(s):=\int_{\R^d}|\hat f(\eta)|\prod_{j=1}^d\min\big(1,|s_j\eta_j|^{-1}\big)\,d\eta,
\qquad s\in\R^d_0,
\end{equation}
one has $\frac1{|R|}\big|\int_Rf\big|\le(2\pi)^{-d}\max(1,\tfrac2\delta)^d\,\mathcal G(a)$
whenever $R=\prod_j[a_j,a_j(1+\delta)]$ with $a\in(0,\infty)^d$, and symmetrically in the other
orthants.
\end{lemma}
\begin{proof}
With the normalization of Section~\ref{SEC:HBV},
$\widehat{\chi_R}(\eta)=\prod_j\frac{e^{ib_j\eta_j}-e^{ia_j\eta_j}}{i\eta_j}$, so
$|\widehat{\chi_R}(\eta)|\le\prod_j\min(b_j-a_j,2/|\eta_j|)$. By H\"older's inequality,
$$
\int|\hat f||\widehat{\chi_R}|\,d\eta\le\|w_{\vec\alpha}\hat f\|_p
\prod_{j=1}^d\Big(\int_\R|\eta_j|^{-\alpha_jp'}\min\big(b_j-a_j,2|\eta_j|^{-1}\big)^{p'}d\eta_j\Big)^{1/p'},
$$
and each factor converges at $0$ because $\alpha_jp'<1$ and at infinity because
$(1+\alpha_j)p'>1$; both are the hypothesis $-1/p<\alpha_j<1-1/p$. Hence
$\hat f\widehat{\chi_R}\in L^1(\R^d)$.

Let $\rho\in\mathcal S(\R^d)$, $\rho\ge0$, $\int\rho=1$, $\rho_\epsilon=\epsilon^{-d}\rho(\cdot/\epsilon)$
and $g_\epsilon:=\chi_R*\rho_\epsilon\in\mathcal S(\R^d)$. Since $f\in L^1$ and $\hat g_\epsilon\in L^1$,
Fubini gives $\int f\overline{g_\epsilon}=(2\pi)^{-d}\int\hat f\overline{\hat g_\epsilon}$ with
$\hat g_\epsilon=\widehat{\chi_R}\hat\rho(\epsilon\cdot)$, $|\hat\rho(\epsilon\eta)|\le1$. As
$\epsilon\to0$, $g_\epsilon\to\chi_R$ a.e.\ with $0\le g_\epsilon\le1$, so the left side tends
to $\int_Rf$ and the right side to $(2\pi)^{-d}\int\hat f\overline{\widehat{\chi_R}}$, both by
dominated convergence. This is \eqref{rev:Fbound}. For the last statement, $b_j-a_j=\delta a_j$,
$|R|=\delta^d\prod a_j$, and
$\frac1{\delta a_j}\min(\delta a_j,2|\eta_j|^{-1})\le\max(1,\frac2\delta)\min(1,\frac1{a_j|\eta_j|})$.
\end{proof}

\begin{lemma}\label{rev:calG-norm}
Let $H$ be measurable on $\R^d$ with $\|w_{\vec\alpha}H\|_p<\infty$, and define
\[
 \mathcal G_H(s)=\int_{\R^d}|H(\eta)|\prod_{j=1}^d\min(1,|s_j\eta_j|^{-1})\,d\eta.
\]
Then $\|w_{\vec\beta}\mathcal G_H\|_p\lesssim\|w_{\vec\alpha}H\|_p$.
In particular this applies to $H=\hat f$ and $\mathcal G_H=\mathcal G$, whether $\hat f$
is a classical or an improper transform; no $L^1$ hypothesis on $f$ is used.
\end{lemma}
\begin{proof}
Split each $\eta_j$-integral in the definition of $\mathcal G_H$ at $|\eta_j|=|s_j|^{-1}$: in the notation of
Lemma~\ref{LM:Hardy-d},
$$
\mathcal G_H(s)=\sum_{S\subseteq[d]}\prod_{j\in S}\frac1{|s_j|}\ T_{\vec\epsilon(S)}\Big(|H|\prod_{j\in S}|\eta_j|^{-1}\Big)(s^{-1}),\qquad\epsilon_j(S)=\mathbf 1_S(j),
$$
the head ($j\notin S$) being the set $|\eta_j|<|s_j|^{-1}$ where the minimum is $1$, and the
tail ($j\in S$) where it is $|s_j\eta_j|^{-1}$. Substituting $\xi=s^{-1}$,
$ds_j=|\xi_j|^{-2}d\xi_j$, the weight becomes $|\xi_j|^{-\beta_jp-2}$, times $|\xi_j|^p$ for
$j\in S$. Lemma~\ref{LM:Hardy-d} applies with $a_j=-\beta_jp-2$ for $j\notin S$ (head; $a_j+1<0$
is $\beta_j>-1/p$; $b_j=a_j+p=\alpha_jp$) and $a_j=p-\beta_jp-2$ for $j\in S$ (tail; $a_j+1>0$
is $\beta_j<1-1/p$; $b_j=2p-2-\beta_jp$, which with the factor $|\eta_j|^{-1}$ gives
$|\eta_j|^{\alpha_jp}$). Each of the $2^d$ terms is bounded by $\|w_{\vec\alpha}H\|_p$.
\end{proof}

\subsection{Parseval's identity for the improper transform}\label{rev:improper}
Lemma~\ref{rev:parseval} was stated for $f\in L^1(\R^d)$. We extend it to the improper
transform $\hat f=\sum_S\hat f_S$ under the finiteness hypothesis
\begin{equation}\label{rev:S}
f\in HBV_0(\R^d_0)\cap L^1_{loc}(\R^d),\qquad \|w_{\vec\beta}V(f;\cdot)\|_{L^p(\R^d)}<\infty .
\end{equation}
Condition \eqref{rev:S} implies \eqref{rev:standing}: by H\"older's inequality the integral in
\eqref{rev:standing} is at most
$$
\big\|w_{\vec\beta}V(f;(a_S,\cdot))\big\|_p\prod_{j\notin S}\big\||x_j|^{-\beta_j}\big\|_{L^{p'}(\mathcal B(b_j))},
$$
which is finite since $\beta_jp'<1$. Fubini gives a finite weighted slice norm for a.e.\ $a_S$.
For arbitrary $a_S>0$, choose such a point $a'_S$ with $0<a'_j<a_j$ in each coordinate.
Since $V(f;(a_S,x_{S^c}))\leq V(f;(a'_S,x_{S^c}))$, the same slice integral is finite.
Thus \eqref{rev:standing} holds for every $a_S$. For $n\in\mathbb N$ put
$f_n:=f\chi_{(-n,n)^d}\in L^1(\R^d)$.

\begin{lemma}\label{rev:trunc}
Let $f$ satisfy \eqref{rev:S}. Then:
\begin{enumerate}[label=$(\roman*)$]
\item $f_n\in HBV_0(\R^d_0)$ and $V(f_n;x)\le2^dV(f;x)$ for all $x,n$;
\item $\hat f_n(\xi)\to\hat f(\xi)$ for every $\xi\in\R^d_0$;
\item $\sup_n\|w_{\vec\alpha}\hat f_n\|_p\le2^dK_0\|w_{\vec\beta}V(f;\cdot)\|_p$;
\item $\|w_{\vec\alpha}\hat f\|_p\le2^dK_0\|w_{\vec\beta}V(f;\cdot)\|_p$ and, for every bounded
box $R$, $\int_Rf=(2\pi)^{-d}\int_{\R^d}\hat f\,\overline{\widehat{\chi_R}}$. In particular the
conclusions of Lemma~\ref{rev:parseval} hold for $f$ and its improper transform.
\end{enumerate}
\end{lemma}
\begin{proof}
$(i)$ Condition $(b)$ of Definition~\ref{DEFI:HBV} is trivial for $f_n$. Fix a corner of
$Q^c(x)$, say $\prod_j[x_j,\infty)$ with $x\in(0,\infty)^d$, and a grid partition of a box
inside it. For a cell $R'=\prod_j[u_j,v_j]$ let $B:=\{j:u_j<n\le v_j\}$ be the set of
coordinates in which the cell crosses the truncation. If some $j$ has $u_j\ge n$ then
$\Delta_{[d]}(f_n;R')=0$; otherwise $\chi_{(-n,n)^d}$ is constant in the coordinates outside
$B$ on $R'$, and expanding the difference gives
$\Delta_{[d]}(f_n;R')=\pm\Delta_{[d]\setminus B}\big(f;R'_{[d]\setminus B};u_B\big)$. Summing
over the cells with a fixed crossing set $B$ --- there is exactly one crossing index per
coordinate of $B$, so the sum is over the cells in the remaining directions --- gives at most
the $([d]\setminus B)$-variation of $f$ over $\prod_{j\notin B}[x_j,\infty)$ at $u_B$, which by
the corner estimate \eqref{EQ:recover-corner} (including the empty mixed difference) is at most the
$[d]$-variation of $f$ over $\prod_{j\notin B}[x_j,\infty)\times\prod_{j\in B}[u_j,\infty)$,
a subset of the corner. There are $2^d$ possible sets $B$, so the $\Delta$-sums over the corner
are at most $2^d$ times the variation of $f$ over that corner; summing over corners gives $(i)$.

$(ii)$ Fix $\xi\in\R^d_0$ and let $n>\max_j|\xi_j|^{-1}$, so that $Q(\xi^{-1})\subset(-n,n)^d$
and $\hat f_{n,\emptyset}=\hat f_\emptyset$. For $S\ne\emptyset$,
$$
\hat f_{n,S}(\xi)-\hat f_S(\xi)=-\int_{Q_S(\xi^{-1})\setminus(-n,n)^d}f(x)e^{ix\cdot\xi}dx,
$$
and $Q_S(\xi^{-1})\setminus(-n,n)^d=\bigcup_{j\in S}\{x\in Q_S(\xi^{-1}):|x_j|\ge n\}$. By
inclusion--exclusion this is a signed sum of $2^{|S|}-1$ regions, each of the form
$Q_S(\zeta^{-1})$ with $\zeta_j^{-1}=n$ for $j$ in a nonempty subset $T\subseteq S$ and
$\zeta_j=\xi_j$ otherwise. Applying Lemma~\ref{LM:ibp} in the variables $S$ at these exterior cutoffs,
and then integrating the variables $S^c$ as in Lemma~\ref{LM:pointwise}, bounds each term by $2^{|S|}\prod_{j\in S}|\xi_j|^{-1}$ times
$\int_{\prod_{j\notin S}\mathcal B(\xi_j^{-1})}V\big(f;(\zeta_S^{-1},x_{S^c})\big)dx_{S^c}$, in
which at least one argument of $V$ equals $n$. The integrand is dominated by
$V(f;(\xi_S^{-1},x_{S^c}))$, integrable by \eqref{rev:standing}, and tends to $0$ as
$n\to\infty$ by Lemma~\ref{LM:Vtail}; dominated convergence shows that each such
term tends to $0$ (for $S=[d]$ there is no integral and the term is $V(f;\zeta^{-1})$ itself).

$(iii)$ By $(i)$ and $V(f_n;\cdot)\le2^dV(f;\cdot)$, the function $f_n$ satisfies
\eqref{rev:standing}; Proposition~\ref{rev:forwardV} applied to $f_n$ (whose improper transform
is the classical one, $f_n$ being integrable) gives
$\|w_{\vec\alpha}\hat f_n\|_p\le K_0\|w_{\vec\beta}V(f_n;\cdot)\|_p\le2^dK_0\|w_{\vec\beta}V(f;\cdot)\|_p$.

$(iv)$ The first assertion follows from $(ii)$, $(iii)$ and Fatou's lemma. For the second, fix
$R$ and $n$ so large that $R\subset(-n,n)^d$; then
$\int_Rf=\int_Rf_n=(2\pi)^{-d}\int\hat f_n\overline{\widehat{\chi_R}}$ by
Lemma~\ref{rev:parseval} applied to $f_n$. Now
$\overline{\widehat{\chi_R}}\in L^{p'}(w_{\vec\alpha}^{-p'})$ by the H\"older computation above,
while by $(ii)$ and $(iii)$ the sequence $(\hat f_n)$ is bounded in $L^p(w_{\vec\alpha}^p)$ and
converges a.e.\ to $\hat f$. This implies weak convergence in that space:
after multiplication by $w_{\vec\alpha}$, boundedness in $L^p$ and a.e.\ convergence
give $L^1$ convergence on every finite-measure set by uniform integrability.
Testing first against bounded functions of finite support identifies the weak
limit; density in $L^{p'}$ and the uniform $L^p$ bound extend the convergence to
every dual test function. Thus
$\int\hat f_n\overline{\widehat{\chi_R}}\to\int\hat f\overline{\widehat{\chi_R}}$. The final
assertion holds because the proof of Lemma~\ref{rev:parseval} used only this identity and the
bound on $|\widehat{\chi_R}|$.
\end{proof}

\begin{rem}\label{rev:S-equiv}
Within $GM(\R^d_0;\mathbf 1,\mathbf 1)$ the finiteness condition in \eqref{rev:S} is equivalent
to $\|w_{\vec\beta}f\|_p<\infty$: indeed $|f|\le V(f;\cdot)$, while conversely
$V(f;x)\le C\Phi(x/m)$ with $\Phi=T_{(1,\dots,1)}\big(|f|/\prod_j|y_j|\big)$, so
Lemma~\ref{LM:Hardy-d} (tail in every variable, $a_j=\beta_jp$, admissible since $\beta_j>-1/p$,
$b_j=\beta_jp+p$) gives
$\|w_{\vec\beta}V(f;\cdot)\|_p\le Cm^\sigma\|w_{\vec\beta}\Phi\|_p\lesssim\|w_{\vec\beta}f\|_p$,
where $\sigma:=\sum_j\beta_j+d/p$.
\end{rem}

\subsection{Oscillation on small boxes}
The next lemma is where the multidimensional structure of the variation enters; it replaces
positivity.

\begin{lemma}\label{rev:osc}
Let $f\in HBV_0(\R^d_0)$ and let $R=\prod_j[a_j,b_j]$ be a box contained in one open orthant.
For $\emptyset\ne A\subseteq[d]$ write $\mathrm{Var}_A(f;R)$ for the $A$-mixed variation of $f$
over $\prod_{j\in A}[a_j,b_j]$ with the remaining coordinates fixed at $a_{A^c}$. Then, for
every $t\in R$,
\begin{equation}\label{rev:osc-eq}
|f(t)|\ \le\ \frac1{|R|}\Big|\int_Rf\Big|+\sum_{\emptyset\ne A\subseteq[d]}\big(2^{|A|}-1\big)\,\mathrm{Var}_A(f;R).
\end{equation}
\end{lemma}
\begin{proof}
For $t,u\in R$ let $\delta_j$ be the one-variable difference replacing $u_j$ by $t_j$. Since
$\prod_j(I+\delta_j)$ evaluated at $u$ gives $f(t)$,
$$
f(t)-f(u)=\sum_{\emptyset\ne A\subseteq[d]}\delta_Af(u),\qquad
|\delta_Af(u)|=\left|\Delta_A\Big(f;\prod_{j\in A}[u_j\wedge t_j,u_j\vee t_j];u_{A^c}\Big)\right|,
$$
so $|\delta_Af(u)|\le\mathrm{Var}_A^{u}(f;R)$, the $A$-mixed variation over $R$ with the other
coordinates at $u_{A^c}$. We compare the base point $u_{A^c}$ with $a_{A^c}$. For a single
coordinate $e\notin A$, a grid partition of $\prod_{j\in A}[a_j,b_j]$ and the corresponding
$A$-differences $g_i$,
$$
\sum_i\big|g_i(u_e,\cdot)\big|\le\sum_i\big|g_i(a_e,\cdot)\big|+\sum_i\big|g_i(u_e,\cdot)-g_i(a_e,\cdot)\big|,
$$
and the last sum consists of $(A\cup\{e\})$-mixed differences of $f$, so it is at most
$\mathrm{Var}^u_{A\cup\{e\}}(f;R)$. Iterating over the coordinates of $A^c$ one obtains
$$
\mathrm{Var}^u_A(f;R)\ \le\ \sum_{B\subseteq A^c}\mathrm{Var}_{A\cup B}(f;R),
$$
all base points now being $a$. Summing over $A$ and collecting terms, the coefficient of
$\mathrm{Var}_{A'}(f;R)$ is the number of pairs $(A,B)$ with $\emptyset\ne A\subseteq A'$ and
$B=A'\setminus A$, that is $2^{|A'|}-1$. Hence
$\sup_{t,u\in R}|f(t)-f(u)|\le\sum_{\emptyset\ne A}(2^{|A|}-1)\mathrm{Var}_A(f;R)$, and
$|f(t)|\le\frac1{|R|}\big|\int_Rf\big|+\frac1{|R|}\int_R|f(t)-f(u)|\,du$ gives
\eqref{rev:osc-eq}.
\end{proof}

\subsection{The main estimate}
For $\lambda>1$ and $x\in\R^d_0$ let $\Sigma_\lambda(x)$ be the sum, over the $2^d$ orthants and
over $\mathbf k\in\mathbb N_0^d$, of the normalized signed averages
$|R_{\mathbf k}|^{-1}\big|\int_{R_{\mathbf k}}f\big|$, where in a given orthant $R_{\mathbf k}$
is the box with $|y_j|\in[\lambda^{k_j}|x_j|,\lambda^{k_j+1}|x_j|]$; for fixed signs these boxes
tile the corresponding corner of $Q^c(x)$. For $T\subseteq[d]$ put
\begin{equation}\label{rev:calV}
\mathcal V_T(x):=\int_{\prod_{j\in T}\mathcal B^c(x_j)}V\big(f;(s_T,x_{T^c})\big)\prod_{j\in T}\frac{ds_j}{|s_j|},
\qquad \mathcal V_\emptyset(x)=V(f;x).
\end{equation}

\begin{lemma}\label{rev:decomp}
Let $f\in HBV_0(\R^d_0)\cap L^1_{loc}(\R^d)$ and $\lambda>1$. Then for every $x\in\R^d_0$
\begin{equation}\label{rev:decomp-eq}
\Phi(x)\ \le\ (\log\lambda)^d\,\Sigma_\lambda(x)\ +\ C_d\sum_{T\subsetneq[d]}(\log\lambda)^{\,d-|T|}\,\mathcal V_T(x),
\end{equation}
with $C_d$ depending only on $d$. Every term of the sum carries at least one factor
$\log\lambda$.
\end{lemma}
\begin{proof}
Fix the corner $C=\prod_j[x_j,\infty)$, $x\in(0,\infty)^d$, and tile it by the boxes
$R_{\mathbf k}$; the other corners are reduced to this one by the reflections $y_j\mapsto-y_j$,
which preserve the relevant integrals and variations after composition of $f$ with the reflection. Since $\int_{R_{\mathbf k}}\prod_j\frac{dt_j}{t_j}=(\log\lambda)^d$,
Lemma~\ref{rev:osc} gives
$$
\int_C\frac{|f(t)|}{t_1\cdots t_d}\,dt\le(\log\lambda)^d\sum_{\mathbf k}\Big[\frac1{|R_{\mathbf k}|}\Big|\int_{R_{\mathbf k}}f\Big|+\sum_{\emptyset\ne A}\big(2^{|A|}-1\big)\mathrm{Var}_A(f;R_{\mathbf k})\Big].
$$
The first sum is the contribution of $C$ to $\Sigma_\lambda(x)$. Fix $A\ne\emptyset$ and sum
over $\mathbf k_A$ first: by additivity of the $A$-variation,
$$
\sum_{\mathbf k_A}\mathrm{Var}_A(f;R_{\mathbf k})\ \le\ V_A\big(f;(x_A,\lambda^{\mathbf k_{A^c}}x_{A^c})\big)\ \le\ g\big(\lambda^{\mathbf k_{A^c}}x_{A^c}\big),
$$
where $g(\sigma_{A^c}):=V(f;(x_A,\sigma_{A^c}))$, by Property~\ref{PROP:recover}; note that
$g$ is non-increasing in each $|\sigma_j|$. For a non-increasing function $h$ of one variable
and $\sigma>0$ one has $\sum_{k\ge0}h(\lambda^k\sigma)\le h(\sigma)+(\log\lambda)^{-1}\int_\sigma^\infty h(u)\frac{du}u$;
applying this in each of the $|A^c|$ variables of $g$ yields
$$
\sum_{\mathbf k_{A^c}}g\big(\lambda^{\mathbf k_{A^c}}x_{A^c}\big)\ \le\ \sum_{T\subseteq A^c}(\log\lambda)^{-|T|}\ \mathcal V_T(x),
$$
Here we bound the positive-ray integrals by the full-ray integrals in \eqref{rev:calV}. The full variation $V$ is used throughout, so summing over corners costs at most an additional factor $2^d$. Multiplying by
$(\log\lambda)^d$ produces the factor $(\log\lambda)^{d-|T|}$, and $T\subseteq A^c\subsetneq[d]$
because $A\ne\emptyset$, so $|T|\le d-1$ and every term retains at least one factor
$\log\lambda$. Summing over $A$ and over the $2^d$ corners gives \eqref{rev:decomp-eq}.
\end{proof}

\begin{thm}\label{rev:main}
Let $d\ge1$, $1<p<\infty$, $-1/p<\alpha_j<1-1/p$, $\beta_j=1-2/p-\alpha_j$, and let
$f\in GM(\R^d_0;\mathbf 1,\mathbf 1)\cap L^1_{loc}(\R^d)$ be real or complex valued. Assume
either
\begin{enumerate}[label=$(\mathrm{\alph*})$]
\item $\|w_{\vec\beta}f\|_p<\infty$ (equivalently, by Remark~\ref{rev:S-equiv}, \eqref{rev:S}), or
\item $f\in L^1(\R^d)$.
\end{enumerate}
Then
\begin{equation}\label{rev:main-eq}
\Big(\int_{\R^d}\prod_j|x_j|^{\beta_jp}|f(x)|^pdx\Big)^{1/p}
\ \lesssim\ \Big(\int_{\R^d}\prod_j|\xi_j|^{\alpha_jp}|\hat f(\xi)|^pd\xi\Big)^{1/p},
\end{equation}
where $\hat f=\sum_S\hat f_S$ (the classical transform if $f\in L^1$), the implied constant
depending only on $d,p,\boldsymbol\alpha$ and the constants $C,m$ of \eqref{rev:GM}.
\end{thm}
\begin{proof}
We may assume $\|w_{\vec\alpha}\hat f\|_p<\infty$. Put $\sigma:=\sum_j\beta_j+d/p$,
$\gamma:=\sum_j(1-1/p-\beta_j)=\sum_j(\alpha_j+1/p)>0$, and
$N(\epsilon):=\|w_{\vec\beta}V(f;\cdot)\|_{L^p(A_\epsilon)}$ for $\epsilon\ge0$, with
$A_\epsilon:=\{x:|x_j|>\epsilon\ \forall j\}$ and $A_0:=\R^d_0$.

\emph{Step 1: a priori bound and Parseval.} Under $(\mathrm a)$, $N(0)<\infty$ by
Remark~\ref{rev:S-equiv} and Lemma~\ref{rev:trunc}$(iv)$ makes Lemma~\ref{rev:parseval}
available for the improper $\hat f$. Under $(\mathrm b)$ with $f\in L^1$,
Lemma~\ref{rev:parseval} applies directly, and on $Q^c(x/m)$ we have
$\prod_j|y_j|\ge m^{-d}\prod_j|x_j|$, so by \eqref{rev:Phi} and \eqref{rev:GM}
\begin{equation}\label{rev:apriori}
V(f;x)\le C\Phi(x/m)\le\frac{Cm^d\|f\|_{L^1}}{|x_1\cdots x_d|};
\end{equation}
since $\beta_jp-p<-1$ this gives
$N(\epsilon)^p\lesssim\|f\|_1^p\prod_j\int_{|x_j|>\epsilon}|x_j|^{\beta_jp-p}dx_j\backsimeq\|f\|_1^p\epsilon^{-\gamma p}$.
Thus $N(\epsilon)\lesssim\|f\|_1\epsilon^{-\gamma}$.

\emph{Step 2: the signed averages.} By Lemma~\ref{rev:parseval} with $\delta=\lambda-1$, in each
orthant $|R_{\mathbf k}|^{-1}\big|\int_{R_{\mathbf k}}f\big|\le c_\lambda\,\mathcal G(\lambda^{\mathbf k}x)$,
so $\Sigma_\lambda(x)\le2^dc_\lambda\sum_{\mathbf k\in\mathbb N_0^d}\mathcal G(\lambda^{\mathbf k}x)$.
By Minkowski's inequality and the changes of variables $y_j=\lambda^{k_j}x_j$,
$$
\|w_{\vec\beta}\Sigma_\lambda\|_p\le2^dc_\lambda\sum_{\mathbf k}\lambda^{-\sum_jk_j(\beta_j+1/p)}\|w_{\vec\beta}\mathcal G\|_p
=2^dc_\lambda\prod_{j=1}^d\big(1-\lambda^{-(\beta_j+1/p)}\big)^{-1}\|w_{\vec\beta}\mathcal G\|_p,
$$
the series converging because $\beta_j+1/p>0$. With Lemma~\ref{rev:calG-norm},
$\|w_{\vec\beta}\Sigma_\lambda\|_p\le K_\lambda\|w_{\vec\alpha}\hat f\|_p$.

\emph{Step 3: the error terms.} Let $T\subsetneq[d]$ and $x\in A_\epsilon$. The integrand of
$\mathcal V_T(x)$ involves $V(f;(s_T,x_{T^c}))$ with $|s_j|>|x_j|>\epsilon$, so the argument
stays in $A_\epsilon$. Extend $V\mathbf1_{A_\epsilon}$ by zero to the full space and apply
the tail case of Lemma~\ref{LM:Hardy-d} ($q=p$) successively in
the variables of $T$, each time with $a=\beta_jp$ --- admissible since $\beta_j>-1/p$, and
$b=\beta_jp+p$, which together with the factor $|s_j|^{-1}$ restores the weight
$|s_j|^{\beta_jp}$. Restricting the output to $A_\epsilon$ gives
$\|w_{\vec\beta}\mathcal V_T\|_{L^p(A_\epsilon)}\le K^{|T|}N(\epsilon)$ with $K$ independent of
$\epsilon\ge0$ and of $T$. Hence, by Lemma~\ref{rev:decomp},
\begin{equation}\label{rev:Phi-norm}
\|w_{\vec\beta}\Phi\|_{L^p(A_\epsilon)}\le(\log\lambda)^dK_\lambda\|w_{\vec\alpha}\hat f\|_p+\theta_0N(\epsilon),
\qquad \theta_0:=C_d\sum_{T\subsetneq[d]}(\log\lambda)^{d-|T|}K^{|T|}.
\end{equation}

\emph{Step 4: absorption.} By \eqref{rev:GM} and the change of variables $x\mapsto mx$,
$N(\epsilon)\le C\|w_{\vec\beta}\Phi(\cdot/m)\|_{L^p(A_\epsilon)}=Cm^\sigma\|w_{\vec\beta}\Phi\|_{L^p(A_{\epsilon/m})}$,
so \eqref{rev:Phi-norm} gives, for every $\epsilon\ge0$,
$$
N(\epsilon)\le a+\theta N(\epsilon/m),\qquad
a:=Cm^\sigma(\log\lambda)^dK_\lambda\|w_{\vec\alpha}\hat f\|_p,\quad \theta:=Cm^\sigma\theta_0 .
$$
Put $M:=\gamma$ under $(\mathrm b)$ and $M:=0$ under
$(\mathrm a)$. Choose $\lambda>1$ so close to $1$ that $\theta m^M<1$; this is possible because every term of $\theta_0$ carries a factor $\log\lambda$
(Lemma~\ref{rev:decomp}), so $\theta_0\to0$ as $\lambda\to1$, while $C,m,K,M,d$ do not depend on
$\lambda$. Under $(\mathrm a)$ take $\epsilon=0$: since $N(0)<\infty$ and $N(0)\le a+\theta N(0)$,
we get $N(0)\le a/(1-\theta)$. Under $(\mathrm b)$ iterate:
$N(\epsilon)\le a\sum_{n<n_0}\theta^n+\theta^{n_0}N(\epsilon/m^{n_0})\le a/(1-\theta)+K\epsilon^{-M}(\theta m^M)^{n_0}$,
and let $n_0\to\infty$, then $\epsilon\to0$. In both cases
$$
\|w_{\vec\beta}V(f;\cdot)\|_{L^p(\R^d)}\le\frac{Cm^\sigma(\log\lambda)^dK_\lambda}{1-\theta}\,\|w_{\vec\alpha}\hat f\|_p,
$$
and $|f(x)|\le V(f;x)$ (Property~\ref{PROP:recover}) gives \eqref{rev:main-eq}.
\end{proof}

\begin{cor}\label{rev:twosided}
Let $d\ge1$, $1<p<\infty$, $-1/p<\alpha_j<1-1/p$, $\beta_j=1-2/p-\alpha_j$, and
$f\in GM(\R^d_0;\mathbf 1,\mathbf 1)\cap L^1_{loc}(\R^d)$.
\begin{enumerate}[label=$(\roman*)$]
\item If $\|w_{\vec\beta}f\|_p<\infty$, then
$\|w_{\vec\alpha}\hat f\|_p\backsimeq\|w_{\vec\beta}f\|_p$, with constants depending only on
$d,p,\boldsymbol\alpha,C,m$.
\item If $f\in L^1(\R^d)$, the same holds as an equivalence of possibly infinite quantities; in
particular $\hat f\in L^p(w_{\vec\alpha}^p)$ if and only if $f\in L^p(w_{\vec\beta}^p)$.
\end{enumerate}
Both statements allow arbitrary real or complex values; no parity or product assumption is imposed.
\end{cor}
\begin{proof}
The forward inequality is Corollary~\ref{COR:classical} with $q=p$ (trivial if
$\|w_{\vec\beta}f\|_p=\infty$; otherwise \eqref{rev:standing} holds by Remark~\ref{REM:finite});
the reverse is Theorem~\ref{rev:main} under $(\mathrm a)$, resp.\ $(\mathrm b)$.
\end{proof}

\begin{rem}[The nonintegrable extension]\label{rev:comments}
For a locally integrable $f$ in the classical GM class, the same reverse estimate
holds conditionally if its improper transform is well defined, the rectangle
Parseval identity in Lemma~\ref{rev:trunc}(iv) holds, and
\[
 N(\epsilon)\le K\epsilon^{-M}<\infty
 \quad\hbox{for all sufficiently small }\epsilon>0
\]
for some finite $K$ and $M\ge0$. Indeed choose the grid ratio so that
$\theta m^M<1$ and iterate as in Step 4. This first establishes $N(0)<\infty$,
possibly with a bound depending on $M$. Now repeat Step 4 at $\epsilon=0$ with
a new grid ratio satisfying merely $\theta<1/2$, chosen using only
$d,p,\boldsymbol\alpha,C,m$. Absorption gives the stated constant independent
of $K,M$. Both the growth bound and the Fourier-pairing justification are needed;
neither is inferred from the desired conclusion. Without additional assumptions,
the implication from finite Fourier norm to finite spatial norm for nonintegrable
data remains open here.
\end{rem}

\begin{cor}[Parameters straddling one]\label{COR:refinedBoas}
Let $1<p<\infty$, $q=p$, $-1/p<\alpha_j<1-1/p$, and $\tau_j\le1\le\kappa_j$ for every $j$.
Let $f$ belong to the corresponding GM class, with $\mathcal R(f)$ defined using \eqref{EQ:beta}.
For finite $\mathcal R(f)$,
\[
 \|w_{\vec\alpha}\widehat f\|_p\asymp
 \|w_{\vec\beta}f\|_p\asymp\mathcal R(f),
 \qquad \beta_j=1-2/p-\alpha_j.
\]
For $L^1$ data these are also equivalent finiteness conditions. The constants
may depend on the parameters and the GM constants.
\end{cor}
\begin{proof}
Property~\ref{PROP:incl} embeds this class in the classical class.
Proposition~\ref{PROP:normequiv} and Corollary~\ref{rev:twosided} give the
finite-norm assertion. For $L^1$ data, a finite Fourier norm first gives a finite
classical spatial norm; the pointwise comparison of weights then gives
$\mathcal R(f)<\infty$, and the same estimates apply.
\end{proof}

\section{A class defined through the primitive}\label{rev:GMbar-sec}
Lemmas~\ref{rev:parseval} and \ref{rev:calG-norm} do not use general monotonicity and apply to
any class in which $V(f;\cdot)$ is controlled by signed averages of $f$. Here is the natural
such class. For $f\in L^1_{loc}(\R^d)$ put $F(s):=\int_{Q(s)}f$ and
\begin{equation}\label{rev:Ibar}
\overline I(f;x):=\int_{\R^d_0}\prod_{j=1}^d\min\Big(\frac1{|s_j|},\frac1{|x_j|}\Big)\frac{|F(s)|}{|s_1\cdots s_d|}\,ds .
\end{equation}

\begin{defi}\label{rev:GMbar}
$f\in HBV_0(\R^d_0)\cap L^1_{loc}(\R^d)$ belongs to $\overline{GM}(\R^d_0)$ if
$V(f;x)\lesssim\overline I(f;x)$ for all $x\in\R^d_0$.
\end{defi}

This is a continuous primitive-controlled condition. Related conditions involving averaged Fourier coefficients are studied in \cite{MN25}; we prove the present continuous statement directly.

\begin{prop}\label{rev:incl}
If $f\ge0$ and $f\in GM(\R^d_0;\mathbf 1,\mathbf 1)\cap L^1_{loc}(\R^d)$,
then $f\in\overline{GM}(\R^d_0)$.
\end{prop}
\begin{proof}
Write $|y_j|^{-1}=\int_{|y_j|}^\infty s_j^{-2}ds_j$ and use Tonelli's theorem ($f\ge0$):
$$
\Phi(x/m)=\int_{(0,\infty)^d}\frac{1}{s_1^2\cdots s_d^2}\int_{\{|x_j|/m<|y_j|<s_j\}}f(y)\,dy\,ds
\le\int_{\{s_j>|x_j|/m\}}\frac{F(s)}{s_1^2\cdots s_d^2}\,ds,
$$
the inner integral vanishing unless $s_j>|x_j|/m$ for all $j$ and being at most $F(s)$. On that
range $s_j^{-1}\le m\min(s_j^{-1},|x_j|^{-1})$, so $\Phi(x/m)\lesssim\overline I(f;x)$, and
$V(f;x)\le C\Phi(x/m)$.
\end{proof}

\begin{thm}\label{rev:GMbar-thm}
Let $1<p<\infty$, $-1/p<\alpha_j<1-1/p$, $\beta_j=1-2/p-\alpha_j$, and
$f\in\overline{GM}(\R^d_0)$. If $\|w_{\vec\beta}f\|_p<\infty$, its improper Fourier transform is well defined and
$\|w_{\vec\alpha}\hat f\|_p\asymp\|w_{\vec\beta}f\|_p$.
If $f\in L^1(\R^d)$, the same estimates characterize finiteness of either norm.
Constants depend only on $d,p,\boldsymbol\alpha$ and the constant in Definition~\ref{rev:GMbar}.
\end{thm}
\begin{proof}
Put $G(s):=|F(s)|/|s_1\cdots s_d|$. Splitting each $s_j$-integral in \eqref{rev:Ibar} at
$|s_j|=|x_j|$,
$$
\overline I(f;x)=\sum_{S\subseteq[d]}\prod_{j\notin S}\frac1{|x_j|}\ T_{\vec\epsilon(S)}\Big(G\prod_{j\in S}|s_j|^{-1}\Big)(x),\qquad \epsilon_j(S)=\mathbf 1_S(j),
$$
and Lemma~\ref{LM:Hardy-d} with $a_j=\beta_jp-p$ for $j\notin S$ (head; $\beta_j<1-1/p$;
$b_j=\beta_jp$) and $a_j=\beta_jp$ for $j\in S$ (tail; $\beta_j>-1/p$; $b_j=\beta_jp+p$, giving
$|s_j|^{\beta_jp}$ with the factor $|s_j|^{-1}$) yields
$\|w_{\vec\beta}V(f;\cdot)\|_p\lesssim\|w_{\vec\beta}\overline I(f;\cdot)\|_p\lesssim\|w_{\vec\beta}G\|_p$.

\emph{Forward.} By Proposition~\ref{rev:forwardV} it suffices that
$\|w_{\vec\beta}V\|_p\lesssim\|w_{\vec\beta}f\|_p$, which follows from the above together with
$G(s)\le|s_1\cdots s_d|^{-1}T_{(0,\dots,0)}|f|(s)$ and Lemma~\ref{LM:Hardy-d} with
$a_j=\beta_jp-p$. The resulting finite weighted variation norm implies
\eqref{rev:standing}, as shown before Lemma~\ref{rev:trunc}, so the improper transform
and the application of Proposition~\ref{rev:forwardV} are justified.

\emph{Reverse.} Lemma~\ref{rev:parseval} is available for $\hat f$: directly if $f\in L^1$, and
by Lemma~\ref{rev:trunc}$(iv)$ if $\|w_{\vec\beta}f\|_p<\infty$, since then
$\|w_{\vec\beta}V\|_p<\infty$ by the forward part. Taking $R=Q(s)$ in
Lemma~\ref{rev:parseval}, so that $b_j-a_j=2|s_j|$ and $|R|=2^d|s_1\cdots s_d|$, gives
$G(s)\le\pi^{-d}\mathcal G(s)$, whence
$\|w_{\vec\beta}f\|_p\le\|w_{\vec\beta}V\|_p\lesssim\|w_{\vec\beta}G\|_p\lesssim\|w_{\vec\beta}\mathcal G\|_p\lesssim\|w_{\vec\alpha}\hat f\|_p$
by Lemma~\ref{rev:calG-norm}.
\end{proof}

\section{One-dimensional specialization and examples}\label{SEC:1d}
For $d=1$, write
\[
 V_f(r)=\operatorname{Var}_{[r,\infty)}f+\operatorname{Var}_{(-\infty,-r]}f,
 \qquad I_f(r)=\int_{|y|\ge r}\frac{|f(y)|}{|y|^\tau+|y|^\kappa}dy.
\]
The condition is $V_f(r)\le CI_f(r/m)$, with $f\in BV_0(\Ro)$.
\begin{cor}[One-dimensional Pitt inequality]\label{COR:1d}
Let $\tau\le\kappa$, $1<p\le q<\infty$, and $-1/q<\alpha<1-1/q$.
For $f\in GM(\Ro;\tau,\kappa)$ with
$R(f)=\||x|^{\beta(x)}f\|_p<\infty$, where
$\beta(x)=2-\alpha-1/p-1/q-\tau$ for $|x|<1$ and
$\beta(x)=2-\alpha-1/p-1/q-\kappa$ otherwise,
$f\in L^1_{loc}(\R)$, its Fourier integral converges on each tail for $\xi\ne0$, and
$\||\xi|^\alpha\widehat f\|_q\lesssim R(f)$.
The constant depends on $p,q,\alpha,\tau,\kappa,C,m$.
\end{cor}
\begin{proof}
Take $d=1$ in Theorem~\ref{TH:main} and Remark~\ref{REM:finite}.
This recovers \cite[Theorem 3.5]{Tok25} under finite spatial norm;
the support specializations are those of Corollary~\ref{COR:support}.
\end{proof}

\begin{rem}[Finite and infinite defining integrals]\label{REM:infinite}
Our definition permits $I_f(r)=\infty$. Since $f$ is bounded on every compact set away from zero, divergence, when present, is at infinity and then occurs for every $r>0$. Such membership is vacuous. It never occurs under the finite-norm hypotheses of Corollary \ref{COR:1d}. For example $f(x)=1/\log(e+|x|)$ is in $BV_0(\Ro)$ and has $I_f(r)=\infty$ for every $\kappa\leq1$. This also shows that neither bounded variation and decay nor GM membership alone implies integrability on the tails. We retain this convention to keep the elementary parameter inclusions valid without auxiliary integrability qualifications, and distinguish finite-integral membership in every classification below.
\end{rem}

\begin{prop}[Rapid decay]\label{PROP:rapid}
If $\kappa>1$ and $f\in GM(\Ro;\tau,\kappa)$, then $|x|^Nf(x)\to0$ as $|x|\to\infty$ for every $N>0$.
\end{prop}
\begin{proof}
Enlarge $m$ to exceed one, which preserves the GM inequality. Since $|f(y)|\leq V_f(|y|)$ and $V_f$ is non-increasing, for $r\geq m$ we have
\[
 V_f(r)\leq C V_f(r/m)\int_{|y|\geq r/m}|y|^{-\kappa}dy
 =A r^{-(\kappa-1)}V_f(r/m),\qquad A=\frac{2Cm^{\kappa-1}}{\kappa-1}.
\]
Fix $r_0\geq m$ and let $r_n=m^nr_0$. Iteration gives
\[
 V_f(r_n)\leq A^n r_0^{-n(\kappa-1)}m^{-(\kappa-1)n(n+1)/2}V_f(r_0).
\]
For $r_n\leq |x|<r_{n+1}$, bound $|x|^N|f(x)|$ by $r_{n+1}^NV_f(r_n)$. The last display tends to zero after multiplication by $r_{n+1}^N$, because its negative quadratic exponent dominates every linear exponent in $n$.
\end{proof}

\begin{prop}[Monotone functions]\label{PROP:monotone1}
Let $M$ be the class of nonnegative functions non-increasing on $(0,\infty)$ and non-decreasing on $(-\infty,0)$, vanishing at both infinities. Then $M\subset GM(\Ro;\tau,\kappa)$ if and only if $\tau=\kappa=1$.
\end{prop}
\begin{proof}
For $f\in M$, $V_f(r)=f(r)+f(-r)$. For any $m>1$,
\[
 \int_{r/m\leq|y|\leq r}\frac{f(y)}{2|y|}\,dy
 \geq\frac{\log m}{2}\bigl(f(r)+f(-r)\bigr).
\]
This proves inclusion at $(1,1)$. If $\kappa>1$, $(1+|x|)^{-1}$ contradicts Proposition \ref{PROP:rapid}. If $\tau<1$, choose $\gamma>\max\{0,1-\tau\}$ and let $f(x)=|x|^{-\gamma}$ for $0<|x|\leq1$, zero otherwise. It is in $M$, $V_f(r)=2r^{-\gamma}$ for $r<1$, and
\[
 I_f(r/m)\leq 2\int_{r/m}^1y^{-\gamma-\tau}dy\lesssim_m r^{1-\gamma-\tau}.
\]
Thus $I_f(r/m)/V_f(r)\to0$. Every ordered pair $\tau\leq\kappa$ other than $(1,1)$ has either $\kappa>1$ or $\tau<1$, completing the proof.
\end{proof}

\begin{ex}[Two powers]\label{EX:powers}
Let $a<0$, $a_0\in\R$, and
\[
 h(x)=\begin{cases}|x|^{a_0},&0<|x|\leq1,\\|x|^a,&|x|>1.\end{cases}
\]
Then $h\in BV_0(\Ro)$, and its tails are integrable precisely when $a<-1$. Under the convention of Remark \ref{REM:infinite}, the complete classification is
\[
 h\in GM(\Ro;\tau,\kappa)\quad\Longleftrightarrow\quad
 \begin{cases}
 \kappa\leq a+1,\quad\text{or}\\
 a+1<\kappa\leq1\ \text{and}\ (a_0\geq0\ \text{or}\ \tau\geq1).
 \end{cases}
\]
The first alternative is vacuous membership. In the second alternative the defining integral is finite.
Indeed, for large $r$, $V_h(r)=2r^a$. The tail integral is infinite if $\kappa\leq a+1$, and otherwise $I_h(r/m)\asymp_m r^{a+1-\kappa}$, forcing $\kappa\leq1$. In the finite-tail case, as $r\downarrow0$,
\[
 V_h(r)\asymp1+r^{a_0},\qquad
 I_h(r/m)\asymp_m
 \begin{cases}
 1,&a_0+1-\tau>0,\\
 1+\log(1/r),&a_0+1-\tau=0,\\
 r^{a_0+1-\tau},&a_0+1-\tau<0.
 \end{cases}
\]
These comparisons prove the remaining condition; on compact intervening intervals the ratio is bounded. In particular, if $a_0<0$, finite-integral membership is possible only for $\tau=\kappa=1$.
\end{ex}

\begin{ex}[Oscillatory tails]\label{EX:osc}
Let $a>b>0$ and
\[
 h_{a,b}(x)=\frac{|\cos(2\pi|x|^b)|}{|x|^a}\,\mathbf1_{\{|x|\geq1\}}.
\]
As $r\to\infty$,
\[
 V_{h_{a,b}}(r)\asymp r^{b-a},\qquad
 I_{h_{a,b}}(r)\asymp r^{1-\kappa-a}\quad(\kappa+a>1).
\]

Set $t=y^b$. Each cosine period contributes variation comparable to
$k^{-a/b}$, by its zero endpoints, an interior maximum, and the derivative bound.
Summation for $k\gtrsim r^b$ gives $r^{b-a}$.
For $I_{h_{a,b}}$, the same substitution gives
\[
 I_{h_{a,b}}(r)\asymp\frac1b\int_{r^b}^\infty
 |\cos(2\pi t)|t^{(1-\kappa-a)/b-1}dt.
\]
Comparison period by period gives the displayed power when $\kappa+a>1$
and divergence otherwise.

Consequently $h_{a,b}\in GM(\Ro;\tau,\kappa)$ exactly when $\kappa\leq1-b$, for any $\tau\leq\kappa$. Its defining integral is finite exactly in the subrange $1-a<\kappa\leq1-b$. For $r$ in a bounded interval the condition follows by choosing $m>1$ and using a fixed positive portion of the tail integral. The tails of $h_{a,b}$ are integrable exactly when $a>1$.

For $0<b<1$ and $a>\max\{b,1\}$, the finite-integral subrange
$1-a<\kappa\le1-b$ supplies integrable examples outside $GM(\Ro;1,1)$.
Also $h_{3,2}\notin GM(\Ro;0,0)$, while
$h_{1+1/n,1}\in GM(\Ro;0,0)$ converges pointwise to a function with infinite
tail variation: uniform controls are essential for limit arguments.

\end{ex}

\begin{ex}[A common element of every parameter class]\label{EX:exp}
The function $h(x)=e^{-|x|}\mathbf1_{\{|x|\geq1\}}$ belongs to $GM(\Ro;\tau,\kappa)$ for every $\tau\leq\kappa$, with finite defining integrals. Fix $m>1$. For $0<r\leq m$, $V_h(r)\leq4e^{-1}$ and
\[
 I_h(r/m)=2\int_1^\infty\frac{e^{-y}}{y^\tau+y^\kappa}dy>0.
\]
For $r>m$, $V_h(r)=2e^{-r}$ and $I_h(r/m)\asymp_m e^{-r/m}r^{-\kappa}$. Their ratio stays bounded. Products of these functions give nonzero elements common to all the multidimensional parameter classes.
\end{ex}

\section{Sharpness and comparison with mixed-tail conditions}\label{SEC:remarks}
\begin{prop}[Optimal Fourier-side range]\label{PROP:sharp}
For every fixed $\vec\tau\leq\vec\kappa$, the range $-1/q<\alpha_i<1-1/q$ in Theorem \ref{TH:main} cannot be enlarged if the assertion is to hold for all members with finite spatial norm.
\end{prop}
\begin{proof}
Take $f(x)=\prod_i\mathbf1_{[1,2]}(x_i)$. Its full mixed variation is bounded, and vanishes if any $|x_i|>2$. Otherwise, with $m=3$, the exterior integration set contains $[1,2]^d$, so $I(f;x/m)$ is bounded below by the fixed positive integral of $1/W$ over that box. Hence $f$ belongs to every parameter class and has finite spatial norm for every finite choice of exponents. Its Fourier transform factors, with each factor having magnitude $2|\sin(\xi_i/2)|/|\xi_i|$. Near zero it tends to one, requiring $\alpha_iq>-1$. At infinity, integration on fixed positive portions of successive sine periods requires $(\alpha_i-1)q<-1$. These are precisely the stated strict bounds; at either endpoint the corresponding series is harmonic. Tonelli's theorem gives the multidimensional conclusion.
\end{proof}

\begin{rem}[Scaling and the spatial exponents]
For classical parameters, Proposition \ref{REM:dilation} preserves the GM constants under independent dilations. Applying a uniform inequality with homogeneous input exponents $\rho_i$ to all coordinate dilates of the box function in Proposition \ref{PROP:sharp} forces
\[
 \rho_i=1-1/p-1/q-\alpha_i.
\]
Indeed, under dilation in coordinate $i$, the Fourier norm scales as $\lambda_i^{-1+1/q+\alpha_i}$ and the spatial norm as $\lambda_i^{-1/p-\rho_i}$. Equality of these powers is necessary as $\lambda_i\to0$ and $\lambda_i\to\infty$. For general parameters our available GM constants vary with the dilation, so this argument does not prove optimality of the piecewise input exponents. We make no such claim.
\end{rem}

\begin{rem}[Relation to the mixed-tail condition]\label{REM:tailcomparison}
In \cite{Tok26}, $g$ is the upper-tail potential of a complex Radon measure $\mu$ on $(0,\infty)^d$. Writing
\[
 U_g(x)=|\mu|\bigl(\prod_i[x_i,\infty)\bigr),
\]
the assumption is $\int_{\prod_i(0,r_i)}U_g\leq K\int_{\prod_i(0,\lambda_i r_i)}|g|$. It leads to mixed-norm inequalities for coordinatewise even extensions. Our exterior-tail condition does not imply that condition: take $g=\prod_i\mathbf1_{[1,2]}$ and its even extension $f$. As in Proposition \ref{PROP:sharp}, $f$ is in every class considered here. On the positive orthant its mixed derivative has total variation $2^d$ above every $x$ with $0<x_i<1$, since it is a product of the signed boundary measures at one and two. For $r_i<1/\lambda_i$, the proposed interior right-hand side is zero and its left-hand side is $2^d\prod_i r_i>0$. Thus no finite $K,\lambda$ can satisfy the mixed-tail condition. This proves a separation in one direction. We do not assert a reverse inclusion or general incomparability without further hypotheses.
\end{rem}

\section*{Declarations}

\noindent\textbf{Author contributions.}
Niyaz Tokmagambetov is the sole author and assumes responsibility for the
research and the final manuscript.

\smallskip
\noindent\textbf{Data and code availability.}
This manuscript presents theoretical results. No datasets were generated or
analysed, and the results do not rely on numerical experiments or research
software.

\smallskip
\noindent\textbf{Ethics approval and consent.}
This work does not involve human participants or animals. Ethical approval
and informed consent are not applicable.

\end{document}